\documentclass[11pt]{amsart}

\usepackage[a4paper,margin=3.0cm]{geometry}
\usepackage{amsmath,amssymb,amsthm,mathtools,xcolor}
\usepackage{enumitem}
\usepackage[hidelinks]{hyperref}
\hypersetup{hypertexnames=false}
\allowdisplaybreaks

\newtheorem{theorem}{Theorem}[section]
\newtheorem{proposition}[theorem]{Proposition}
\newtheorem{lemma}[theorem]{Lemma}

\theoremstyle{remark}
\newtheorem{remark}[theorem]{Remark}

\newcommand{\F}{\mathbb F}
\newcommand{\K}{\mathbb K}
\newcommand{\PP}{\mathbb P}
\newcommand{\cS}{\mathcal S}
\newcommand{\cR}{\mathcal R}

\newcommand{\Aut}{\operatorname{Aut}}
\newcommand{\ord}{\operatorname{ord}}
\newcommand{\Sz}{\operatorname{Sz}}
\newcommand{\Ree}{\operatorname{Ree}}
\newcommand{\PGL}{\operatorname{PGL}}
\newcommand{\GL}{\operatorname{GL}}
\newcommand{\PGU}{\operatorname{PGU}}

\title[Explicit Suzuki and Ree invariants]{Explicit invariants of the Suzuki and Ree groups in their function fields}
\author{Marco Timpanella}
\address{Dipartimento di Matematica e Informatica, Universit\`a degli Studi di Perugia, Via Vanvitelli 1, 06123 Perugia, Italy}
\email{marco.timpanella@unipg.it}

\subjclass[2020]{11G20, 14H37, 14H05, 20G40}
\keywords{Suzuki curve, Ree curve, Dickson invariant, function field, quotient curve, Suzuki group, Ree group}

\begin{document}

\begin{abstract}
The Suzuki and Ree curves are two classical Deligne--Lusztig curves.  They are maximal over suitable finite fields and their full automorphism groups are the Suzuki and Ree groups.  In this paper we determine explicit generators of the fixed fields of these full automorphism groups in the corresponding function fields. Motivated by Dickson's classical study of invariants of finite linear
groups, we use the natural projective representations of the two groups.
In the Suzuki case the proof uses restricted Dickson invariants, whereas
in the Ree case it relies on an invariant bilinear form of the
seven-dimensional representation.
\end{abstract}

\maketitle

\section{Introduction}

Let \(\K\) be an algebraically closed field of characteristic \(p>0\), and let
\(F|\K\) be a function field of one variable.  If \(G\) is a finite group of
\(\K\)-automorphisms of \(F\), the fixed field
\[
        F^G=\{u\in F:\ g(u)=u\text{ for every }g\in G\}
\]
is the function field of the quotient curve.  When \(F^G\) is rational, the
quotient is generated by one element.  The problem considered in this paper is
to determine such a generator explicitly, as a rational function in the chosen
coordinates of the original curve.

This is a concrete invariant-theoretic problem.  In its classical form, it goes
back to Dickson's work on the invariants of finite linear groups.  For instance,
if \(G=\PGL(2,q)\) acts on the rational function field \(\K(x)\), then the fixed
field is generated by the Dickson-type invariant
\[
        \frac{(x^{q^2}-x)^{q+1}}{(x^q-x)^{q^2+1}}.
\]
More generally, Dickson determined fundamental invariants for finite general
linear groups in positive characteristic.  These formulas are classical, but
they also show that explicit invariant computations are usually quite rigid:
one has to find a rational expression which is fixed by a large finite group and
then prove that it has the correct degree.

For function fields of positive genus the same question is usually much more
difficult.  Even when the quotient \(F^G\) is known to be rational by general
results, it may be hard to write a generator of \(F^G\) in the given affine
coordinates of the curve.  The difficulty is not only computational.  The
automorphism group may act through non-linear transformations on the affine
coordinates, and a priori there is no reason for the generator of the quotient
to have a simple expression.  For this reason, explicit closed formulas for
fixed fields of large automorphism groups of algebraic curves are rather rare.

Dickson invariants already have a notable history in the study of curves over
finite fields.  A particularly relevant example is the
Dickson--Guralnick--Zieve curve.  It is the plane curve defined by the vanishing
of a quotient of two Moore determinants, which is a Dickson invariant for
\(\PGL(3,q)\), and its full automorphism group is \(\PGL(3,q)\); see
\cite{GKTDGZ}.  The same curve is a basic example of a plane curve which is
Frobenius nonclassical with respect to both \(\F_q\) and \(\F_{q^3}\), within
the theory of multi-Frobenius nonclassical curves developed by Borges
\cite{Borges}; see also \cite{BorgesFukasawa} for its relation with Galois
points.  Its large automorphism group, quotient curves, rational points, and
associated complete arcs show that a Dickson invariant can encode at the same
time group-theoretic, arithmetic, and finite-geometric information.

A recent example in this direction is due to Gatti, Ghiandoni and Korchm\'aros,
who determined an explicit invariant of the full automorphism group
\(\PGU(3,q)\) in the Hermitian function field.  Their construction uses
Dickson invariants of \(\PGL(3,q^2)\) and restricts them to the Hermitian curve.
Thus the Hermitian case gives a model strategy: one starts from a classical
linear invariant in projective space and then restricts it to a special curve
which carries a natural projective representation of its automorphism group.
The resulting generator can then be used as a coordinate on the base of the
Galois covering and can assist in obtaining equations for intermediate
subfields by elimination \cite{GGK}.  This complements the extensive study of
subfields and quotient curves of the Hermitian function field; see, for
instance, \cite{GarciaStichtenothXing}.

The purpose of this paper is to carry out the corresponding computation for
the Suzuki and Ree function fields.  Together with the Hermitian curve, the
Suzuki and Ree curves are the standard Deligne--Lusztig curves associated with
the groups of types \({}^2B_2\) and \({}^2G_2\); see
\cite{DeligneLusztig,Hansen}.  They play a central role in the theory of
algebraic curves with many rational points over finite fields and provide
explicit geometric realizations of two families of finite simple groups.
Recall that a curve \(\mathcal X\) defined over a finite field \(\F_Q\) is
called maximal over \(\F_Q\) if it attains the Hasse--Weil upper bound
\[
        |\mathcal X(\F_Q)|=Q+1+2g\sqrt Q,
\]
where \(g\) is the genus of \(\mathcal X\).  The Suzuki curve is maximal over
\(\F_{q^4}\), while the Ree curve is maximal over \(\F_{q^6}\).  Their full
automorphism groups are the Suzuki group \(\Sz(q)\) and the Ree group
\({}^2G_2(q)\), respectively.

These curves are important not only because of their large automorphism groups,
but also because of their many quotients and subcovers.  A curve covered by a
maximal curve through a covering defined over the base field is again maximal,
and therefore such coverings provide a systematic source of examples of maximal
function fields.  The quotient curves of the Suzuki
curve were studied systematically in \cite{GKT}, while subfields of the Ree
function field and their rational places were investigated in
\cite{CakOzb,CakOzbPoints}.  More recent work on cyclic covers and subcovers of
the Suzuki and Ree curves has produced further maximal curves and new genera in
the spectrum of maximal curves \cite{GiuliettiMontanucciQuoosZini,Skabelund}.
These results also underline that the Suzuki and Ree families are not merely
variants of the Hermitian case; some of them are not Galois covered by the
corresponding Hermitian curve \cite{MontanucciZini}.

A further motivation comes from the construction of covers of higher genus
on which the full automorphism group of the original curve lifts.
The cyclic covers constructed by Skabelund provide significant examples:
their automorphism groups are direct products of lifted copies of the full
Suzuki and Ree groups with the cyclic covering groups
\cite{Skabelund,GiuliettiMontanucciQuoosZini}.  The explicit generators obtained here provide a different way to
construct covers retaining the action of the full group, including
covers of positive $p$-rank; see Remark~\ref{rem:prank}.

Explicitness is essential in these applications.  A generator of a fixed field
gives the quotient map itself and provides a starting point for equations of
intermediate quotients, for the study of ramification and divisors, and for
computations of Riemann--Roch spaces.  These data are also relevant to
algebraic-geometry codes.  Deligne--Lusztig varieties were connected with group
codes already in \cite{Hansen}.  AG codes arising from the Hermitian and Suzuki
curves were studied, for instance, in \cite{KorchmarosNagyTimpanella,Matthews};
codes carrying the Suzuki group as a group of automorphisms were constructed in
\cite{EidHassonKsirPeachey}; and classical, quantum, and convolutional codes
from cyclic extensions of the Suzuki and Ree curves were studied in
\cite{MontanucciTimpanellaZini}.  Further constructions include AG and quantum
codes from a generalization of the Deligne--Lusztig curve of Suzuki type
\cite{TimpanellaGeneralization}, and two-point AG codes from a Skabelund maximal
curve covering the Suzuki curve \cite{LandiTimpanellaVicino}.  Thus an explicit
description of the quotient is useful both for the arithmetic study of the
curves and for constructions based on their symmetries.

The main results of this paper give explicit generators for the fixed fields of
the full Suzuki and Ree groups.  In characteristic \(2\), let
\[
        q_0=2^s,\qquad q=2q_0^2,
\]
and let \(\K(\cS_q)\) be the Suzuki function field.  We prove that
\(\K(\cS_q)^{\Sz(q)}\) is generated by
\[
        t_S=
        \frac{B_S(\theta)^{q-2q_0+1}}{\theta^{q^2}},
        \qquad
        \theta=(x^q+x)^{q-1},
\]
where
\[
        B_S(T)=1+T^q+T^{q+2q_0}+T^{q+2q_0+1}.
\]
In characteristic \(3\), let
\[
        q_0=3^s,\qquad q=3q_0^2,
\]
and let \(\K(\cR_q)\) be the Ree function field.  We prove that
\(\K(\cR_q)^{\Ree(q)}\) is generated by
\[
        t_R=
        \frac{B_R(\vartheta)^{q-3q_0+1}}{\vartheta^{q^3}},
        \qquad
        \vartheta=(x^q-x)^{q-1},
\]
where \[
\begin{aligned}
B_R(T)={}&1+T^{q^2}+T^{q^2+3q_0q}
          -T^{q^2+3q_0q+q}+T^{q^2+3q_0q+2q}\\
&+T^{q^2+3q_0q+2q+3q_0}
 +T^{q^2+3q_0q+2q+3q_0+1}.
\end{aligned}
\]

The two cases are approached through their natural ovoid representations.  For
the Suzuki group, restricting a Dickson quotient from \(\PGL(4,q)\) produces
the candidate, and the homogeneity of one restricted Dickson invariant gives the
transformation formula required for the involution.  The degree of the resulting
function is then equal to \( |\Sz(q)| \), so Artin's theorem identifies the
fixed field.  For the Ree group, a direct restriction of Dickson invariants is
less effective.  Instead, the representation in \(\PP^6\) and the pole orders
of its coordinates lead to the analogous formula.  The key identity for the
Ree involution is proved by using the quadratic form preserved by the
seven-dimensional representation and a sequence of Frobenius pairings.  This
difference between the two proofs reflects the main difficulty of the problem:
the projective representation suggests the shape of an invariant, but the
geometry of the curve determines how it can actually be reduced to a usable
formula.  In both cases the degree of the resulting function is the order of the
corresponding group.

The paper is organized as follows.  Section~2 recalls the preliminaries on
the Suzuki curve, the Ree curve, and Dickson invariants.  Section~3 proves
the formula for the Suzuki invariant.  Section~4 proves the corresponding
formula for the Ree invariant and closes with an application of the two
invariants to the construction of covers with prescribed symmetries.

\section{Preliminaries}

\subsection{Preliminaries on the Suzuki curve}

In this section we recall the standard facts on the Suzuki function field that will be used later; see for instance \cite[Chapter 12]{HKT} and \cite{GKT}.  Here $p=2$.  Let
\[
        q_0=2^s,
        \qquad
        q=2q_0^2,
        \qquad s\geq 1.
\]
The Suzuki curve $\cS_q$ is the nonsingular model of
\[
        y^q+y=x^{q_0}(x^q+x).
\]
We write $\K(\cS_q)=\K(x,y)$.  Its full automorphism group is
\[
        \Aut(\cS_q)\cong\Sz(q),
\]
and
\[
        |\Sz(q)|=q^2(q-1)(q^2+1).
\]
The quotient $\cS_q/\Sz(q)$ is rational.

We use the standard auxiliary functions
\[
        z=x^{2q_0+1}+y^{2q_0},\qquad w=xy^{2q_0}+z^{2q_0}.
\]
They satisfy
\begin{equation}\label{eq:suzuki-basic-relations}
        z^q+z=x^{2q_0}(x^q+x),
        \qquad
        w^q+w=y^{2q_0}(x^q+x),
\end{equation}
and
\[
        y=x^{q_0+1}+z^{q_0},\qquad w=x^{2q_0+2}+xz+z^{2q_0}=y^2+xz;
\]
see \cite[Section~3]{GKT}.

Following the standard projective description of the Suzuki--Tits ovoid, we use the coordinates
\[
        (1:x:z:w)\in\PP^3,\qquad w=xz+x^{2q_0+2}+z^{2q_0}.
\]
More precisely, the Suzuki--Tits ovoid is
\[
        \Gamma=
        \left\{(1:a:c:ac+a^{2q_0+2}+c^{2q_0}):a,c\in\F_q\right\}
        \cup\{(0:0:0:1)\},
\]
and $\Sz(q)$ is realized as the subgroup of $\PGL(4,q)$ preserving $\Gamma$; see \cite[Section~2.1]{EidDuursma}.
In the function field we use the same coordinate functions $(1:x:z:w)$.
The following generators, in these coordinates, are standard; see
\cite[Section~2]{GKT}.  For $a,c\in\F_q$, let
\[
\tau_{a,c}:
\begin{cases}
        x\mapsto x+a,\\
        z\mapsto z+a^{2q_0}x+c,\\
        w\mapsto w+az+(a^{2q_0+1}+c)x+a^{2q_0+2}+ac+c^{2q_0}.
\end{cases}
\]
For $d\in\F_q^*$, let
\[
\psi_d:
\begin{cases}
        x\mapsto dx,\\
        z\mapsto d^{2q_0+1}z,\\
        w\mapsto d^{2q_0+2}w.
\end{cases}
\]
Finally, let
\[
\iota:
\begin{cases}
        x\mapsto z/w,\\
        z\mapsto x/w,\\
        w\mapsto 1/w.
\end{cases}
\]
Equivalently,
\[
        \iota:(1:x:z:w)\mapsto (w:z:x:1).
\]
The transformations $\tau_{a,c}$ and $\psi_d$ generate the stabilizer of
$P_\infty$, and this stabilizer together with $\iota$ generates $\Sz(q)$.

\subsection{Preliminaries on the Ree curve}

We now recall the corresponding facts for the Ree function field; see \cite{Pedersen,Skabelund,EidDuursma}.  Here $p=3$.  Let
\[
        q_0=3^s,
        \qquad
        q=3q_0^2,
        \qquad s\geq 1.
\]
The Ree curve $\cR_q$ is the nonsingular model of
\begin{equation}\label{eq:ree-curve}
        y^q-y=x^{q_0}(x^q-x),\qquad z^q-z=x^{2q_0}(x^q-x).
\end{equation}
We put
\[
        A=x^q-x.
\]
Then
\[
        y^q-y=x^{q_0}A,
        \qquad
        z^q-z=x^{2q_0}A.
\]
The full automorphism group is
\[
        \Aut(\cR_q)\cong\Ree(q)={}^2G_2(q),
\]
and
\[
        |\Ree(q)|=q^3(q-1)(q^3+1).
\]
The quotient $\cR_q/\Ree(q)$ is rational.

We use the following auxiliary functions:
\[
\begin{aligned}
 w_1&=x^{3q_0+1}-y^{3q_0},
& w_2&=xy^{3q_0}-z^{3q_0},
& w_3&=xz^{3q_0}-w_1^{3q_0},\\
 w_4&=xw_2^{q_0}-yw_1^{q_0},
& v&=xw_3^{q_0}-zw_1^{q_0},
& w_5&=yw_3^{q_0}-zw_2^{q_0},\\
 w_6&=v^{3q_0}-w_2^{3q_0}+xw_4^{3q_0},
& w_7&=w_2+v,
& w_8&=w_5^{3q_0}+xw_7^{3q_0},\\
 w_9&=w_4w_2^{q_0}-yw_6^{q_0},
& w_{10}&=zw_6^{q_0}-w_3^{q_0}w_4.
\end{aligned}
\]
For $w_5$ and $w_8$ we use throughout the convention displayed above.  This is
the convention determined by equations (A.57) and (A.44) of
\cite{EidDuursma} and is compatible with the projective coordinates and the
identities used below.  Notice that the earlier recursive formulas (4.9) and
(4.12) in the same reference contain different subscripts for these two
functions.
These functions satisfy the Frobenius-difference identities
\begin{align}\label{eq:ree-frob-differences}
 w_1^q-w_1&=x^{3q_0}A,
& w_2^q-w_2&=y^{3q_0}A,
& w_3^q-w_3&=z^{3q_0}A,\notag\\
 w_6^q-w_6&=w_4^{3q_0}A,
& w_8^q-w_8&=w_7^{3q_0}A.
\end{align}
These identities are recorded explicitly in \cite{Skabelund,EidDuursma}.

The stabilizer of $P_\infty$ in $\Ree(q)$ consists of transformations
\[
\begin{cases}
        x\mapsto ax+b,\\
        y\mapsto a^{q_0+1}y+ab^{q_0}x+c,\\
        z\mapsto a^{2q_0+1}z-a^{q_0+1}b^{q_0}y+ab^{2q_0}x+d,
\end{cases}
\]
where $a\in\F_q^*$ and $b,c,d\in\F_q$.  The group is generated by this stabilizer and by the involution
\[
        \phi(x)=\frac{w_6}{w_8},\qquad \phi(y)=\frac{w_{10}}{w_8},\qquad \phi(z)=\frac{w_9}{w_8}.
\]
The Ree ovoid representation that we use is
\[
        (1:x:w_1:w_2:w_3:w_6:w_8)\in\PP^6.
\]
The seven-dimensional space spanned by these functions is preserved by $\Ree(q)$ and gives the standard projective representation described in \cite[Remark~7.4]{EidDuursma}.  Under $\phi$ the coordinates transform as
\[
\begin{aligned}
        \phi(x)&=\frac{w_6}{w_8},
&       \phi(w_1)&=\frac{w_3}{w_8},
&       \phi(w_2)&=-\frac{w_2}{w_8},\\
        \phi(w_3)&=\frac{w_1}{w_8},
&       \phi(w_6)&=\frac{x}{w_8},
&       \phi(w_8)&=\frac1{w_8}.
\end{aligned}
\]
These formulas belong to the standard explicit description of the Ree automorphism group and its projective coordinates; see \cite{Pedersen,Skabelund,EidDuursma}.

The pole orders at $P_\infty$ of these seven functions are
\begin{equation}\label{eq:ree-pole-orders}
\begin{aligned}
-\ord_{P_\infty}(1)&=0, & -\ord_{P_\infty}(x)&=q^2,\\
-\ord_{P_\infty}(w_1)&=q^2+3q_0q, & -\ord_{P_\infty}(w_2)&=q^2+3q_0q+q,\\
-\ord_{P_\infty}(w_3)&=q^2+3q_0q+2q, & -\ord_{P_\infty}(w_6)&=q^2+3q_0q+2q+3q_0,\\
-\ord_{P_\infty}(w_8)&=q^2+3q_0q+2q+3q_0+1.
\end{aligned}
\end{equation}
These orders are listed in \cite[Table~14]{EidDuursma}.  For later use, set
\begin{equation}\label{eq:ree-E}
\begin{aligned}
E_0&=0, & E_1&=q^2,\\
E_2&=q^2+3q_0q, & E_3&=q^2+3q_0q+q,\\
E_4&=q^2+3q_0q+2q, & E_5&=q^2+3q_0q+2q+3q_0,\\
E_6&=q^2+3q_0q+2q+3q_0+1.
\end{aligned}
\end{equation}
Thus $E_i=-\ord_{P_\infty}(f_i)$ for
\[
(f_0,f_1,f_2,f_3,f_4,f_5,f_6)=(1,x,w_1,w_2,w_3,w_6,w_8).
\]
If $\mu=q-3q_0+1$, then
\[
        \mu E_6=(q-3q_0+1)(q^2+3q_0q+2q+3q_0+1)=q^3+1.
\]
This identity will determine the degree of the Ree invariant in Section~4.

\subsection{Preliminaries on Dickson invariants}
We recall the form of the Dickson invariants used below; see \cite{Dickson} and, for the same determinant notation in the context of function fields, \cite[Section~2]{GGK}.  Let $q$ be a power of the characteristic $p$ of $\K$.  Let $X_0,\ldots,X_r$ be homogeneous coordinates on $\PP^r$.  Consider the Moore matrix
\[
        M_r(X)=
        \begin{pmatrix}
        X_0 & X_0^q & X_0^{q^2} & \cdots & X_0^{q^{r+1}}\\
        X_1 & X_1^q & X_1^{q^2} & \cdots & X_1^{q^{r+1}}\\
        \vdots & \vdots & \vdots & & \vdots\\
        X_r & X_r^q & X_r^{q^2} & \cdots & X_r^{q^{r+1}}
        \end{pmatrix}.
\]
For $0\le i\le r+1$, let $D_i(X)$ be the determinant obtained from $M_r(X)$ by deleting the column with exponent $q^i$.  Dickson proved that $D_{r+1}$ divides $D_i$ in $\K[X_0,\ldots,X_r]$ for $0\le i\le r$ and that the quotients
\[
        C_i=\frac{D_i}{D_{r+1}},\qquad 0\le i\le r,
\]
are invariant under $\GL(r+1,q)$.  Indeed, if $M\in\GL(r+1,q)$, then the entries of $M$ are fixed by the $q$-Frobenius and hence
\[
        D_i(MX)=\det(M)D_i(X).
\]
Thus each ratio $D_i/D_{r+1}$ is unchanged under the action of $\GL(r+1,q)$.  The degree of $C_i$ is
\[
        \deg C_i=q^{r+1}-q^i.
\]
We will use only the following degree-zero rational combinations of these invariants.

For the Suzuki ovoid representation we use $r=3$.  The Moore matrix is
\[
M_3(X)=
\begin{pmatrix}
X_0 & X_0^q & X_0^{q^2} & X_0^{q^3} & X_0^{q^4}\\
X_1 & X_1^q & X_1^{q^2} & X_1^{q^3} & X_1^{q^4}\\
X_2 & X_2^q & X_2^{q^2} & X_2^{q^3} & X_2^{q^4}\\
X_3 & X_3^q & X_3^{q^2} & X_3^{q^3} & X_3^{q^4}
\end{pmatrix}.
\]
The rational function used in the Suzuki case is
\begin{equation}\label{eq:suzuki-dickson-quotient}
        U_S=\frac{C_3^{q+1}}{C_2^q}=
        \frac{D_3^{q+1}}{D_2^qD_4}.
\end{equation}

For the Ree representation, the direct analogue of the Dickson quotient used
in the Suzuki case is, with $r=6$,
\[
        U_R=\frac{C_6^{q+1}}{C_5^q},
        \qquad
        \alpha=(1,x,w_1,w_2,w_3,w_6,w_8),
\]
where $C_5$ and $C_6$ are the Dickson invariants of $\GL(7,q)$.
By \eqref{eq:ree-pole-orders}, the pole orders of the coordinates of
$\alpha$ are strictly increasing.  Hence, in each of
$D_5(\alpha)$, $D_6(\alpha)$ and $D_7(\alpha)$, the determinant
expansion contains a unique term of maximal pole order.  In particular,
these determinants are nonzero.  Thus $U_R$ restricts to a well-defined
$\Ree(q)$-invariant rational function on $\cR_q$ and, by
Theorem~\ref{thm:ree-invariant} below, belongs to $\K(t_R)$.
In contrast with the Suzuki case,
however, we do not obtain a simple reduction of this quotient giving $t_R$
explicitly.

The seven-dimensional Ree representation preserves an additional bilinear
form.  Using the same Frobenius conjugates that occur in the Moore determinants
defining the Dickson invariants, this form gives the identity
\[
        B_R(\vartheta)=
        \frac{\langle\alpha,\alpha^{q^4}\rangle}
             {\langle\alpha,\alpha^{q^3}\rangle}.
\]
It will be proved in \eqref{eq:ree-B-pairing}.  Thus the Ree calculation starts
from the same Dickson framework as the Suzuki calculation, but uses the
additional bilinear structure of the Ree representation to obtain an explicit
generator.

\section{An explicit invariant for the Suzuki group}

Set
\[
        A=x^q+x,
        \qquad
        \theta=A^{q-1}.
\]
Put
\[
        m=q-2q_0+1,\qquad n=q+2q_0+1.
\]
Since $q=2q_0^2$, one has
\[
        mn=q^2+1.
\]
Let
\[
        B_S(T)=1+T^q+T^{q+2q_0}+T^{q+2q_0+1}.
\]
We define
\begin{equation}\label{eq:suzuki-t}
        t_S=\frac{B_S(\theta)^m}{\theta^{q^2}}.
\end{equation}
We prove in this section that
\[
        \K(\cS_q)^{\Sz(q)}=\K(t_S),
\]
and hence $t_S$ is an invariant of the Suzuki group in the Suzuki function field.
We first explain where the formula defining $t_S$ comes from.  Consider the rational function $U_S$ in \eqref{eq:suzuki-dickson-quotient}.  Since $U_S$ is invariant under $\PGL(4,q)$ and $\Sz(q)$ acts through the Suzuki--Tits ovoid representation in $\PP^3$, the restriction
\[
        U_S(1,x,z,w)=\frac{D_3(1,x,z,w)^{q+1}}{D_2(1,x,z,w)^qD_4(1,x,z,w)}
\]
is fixed by $\Sz(q)$.  This function is not the desired generator, but it is a $q$-th power of it.

Indeed, in the Suzuki coordinates $(1:x:z:w)$ the determinants entering $U_S$ are
\begin{align*}
D_4(1,x,z,w)&=
\det\begin{pmatrix}
1&1&1&1\\
x&x^q&x^{q^2}&x^{q^3}\\
z&z^q&z^{q^2}&z^{q^3}\\
w&w^q&w^{q^2}&w^{q^3}
\end{pmatrix},\\[1ex]
D_3(1,x,z,w)&=
\det\begin{pmatrix}
1&1&1&1\\
x&x^q&x^{q^2}&x^{q^4}\\
z&z^q&z^{q^2}&z^{q^4}\\
w&w^q&w^{q^2}&w^{q^4}
\end{pmatrix},\\[1ex]
D_2(1,x,z,w)&=
\det\begin{pmatrix}
1&1&1&1\\
x&x^q&x^{q^3}&x^{q^4}\\
z&z^q&z^{q^3}&z^{q^4}\\
w&w^q&w^{q^3}&w^{q^4}
\end{pmatrix}.
\end{align*}
Put
\[
        R=(1,x^{2q_0},y^{2q_0}).
\]
By \eqref{eq:suzuki-basic-relations},
\[
        (x,z,w)^q+(x,z,w)=AR,
\]
and therefore, for every $h\geq1$,
\begin{equation}\label{eq:suzuki-iteration}
        (x,z,w)^{q^h}+(x,z,w)
        =\sum_{j=0}^{h-1}A^{q^j}R^{q^j}.
\end{equation}
For $0\leq i<j<k$, set
\[
\Delta_{ijk}:=
\det\begin{pmatrix}
1&1&1\\
x^{q^i}&x^{q^j}&x^{q^k}\\
y^{q^i}&y^{q^j}&y^{q^k}
\end{pmatrix}.
\]
Subtracting the first column from the remaining columns and using
\eqref{eq:suzuki-iteration}, elementary column operations give
\begin{equation}\label{eq:suzuki-D-reduction}
\begin{aligned}
D_4(1,x,z,w)&=A^{1+q+q^2}\Delta_{012}^{2q_0},\\
D_3(1,x,z,w)&=A^{1+q+q^2}\Delta_{012}^{2q_0}+A^{1+q+q^3}\Delta_{013}^{2q_0},\\
D_2(1,x,z,w)&=A^{1+q+q^3}\Delta_{013}^{2q_0}+A^{1+q^2+q^3}\Delta_{023}^{2q_0}.
\end{aligned}
\end{equation}
To evaluate the three determinants, we use, for $h\geq1$,
\begin{equation}\label{eq:suzuki-xy-iterations}
x^{q^h}+x=\sum_{j=0}^{h-1}A^{q^j},\qquad y^{q^h}+y=\sum_{j=0}^{h-1}x^{q_0q^j}A^{q^j}.
\end{equation}
For $\Delta_{012}$, subtracting the first column from the other two gives
\[
\begin{aligned}
\Delta_{012}
&=\det\begin{pmatrix}
1&0&0\\
x&A&x^{q^2}+x\\
y&x^{q_0}A&y^{q^2}+y
\end{pmatrix}\\
&=A(y^{q^2}+y)+(x^{q^2}+x)(y^q+y)\\
&=A\bigl(x^{q_0}A+x^{q_0q}A^q\bigr)+(A+A^q)x^{q_0}A\\
&=A^{q+1}\bigl(x^{q_0q}+x^{q_0}\bigr)=A^{q+q_0+1}.
\end{aligned}
\]
The same procedure for the exponents $0,1,3$ gives
\[
\begin{aligned}
\Delta_{013}
&=A(y^{q^3}+y)+(x^{q^3}+x)(y^q+y)\\
&=A\bigl(x^{q_0}A+x^{q_0q}A^q+x^{q_0q^2}A^{q^2}\bigr)
 +(A+A^q+A^{q^2})x^{q_0}A\\
&=A^{q+1}\bigl(x^{q_0q}+x^{q_0}\bigr)
 +A^{q^2+1}\bigl(x^{q_0q^2}+x^{q_0}\bigr)\\
&=A^{q+q_0+1}+A^{q^2+q_0+1}+A^{q^2+qq_0+1}.
\end{aligned}
\]
Finally,
\[
\begin{aligned}
\Delta_{023}
&=(x^{q^2}+x)(y^{q^3}+y)+(x^{q^3}+x)(y^{q^2}+y)\\
&=(A+A^q)x^{q_0q^2}A^{q^2}+A^{q^2}\bigl(x^{q_0}A+x^{q_0q}A^q\bigr)\\
&=A^{q^2+1}\bigl(x^{q_0q^2}+x^{q_0}\bigr)
 +A^{q^2+q}\bigl(x^{q_0q^2}+x^{q_0q}\bigr)\\
&=A^{q^2+q_0+1}+A^{q^2+qq_0+1}+A^{q^2+qq_0+q}.
\end{aligned}
\]
Since $\theta=A^{q-1}$, these formulas imply
\[
\frac{\Delta_{013}}{\Delta_{012}}=1+\theta^q+\theta^{q+q_0},\qquad
\frac{\Delta_{023}}{\Delta_{012}}=\theta^q+\theta^{q+q_0}+\theta^{q+q_0+1}.
\]
Substitution into \eqref{eq:suzuki-D-reduction} gives
\[
\begin{aligned}
C_3(1,x,z,w)
&=1+\theta^{q^2}\left(\frac{\Delta_{013}}{\Delta_{012}}\right)^{2q_0}\\
&=1+\theta^{q^2}\left(1+\theta^q+\theta^{q+q_0}\right)^{2q_0}\\
&=1+\theta^{q^2}+\theta^{q^2+2q_0q}+\theta^{q^2+2q_0q+q}=B_S(\theta)^q.
\end{aligned}
\]
Similarly,
\[
\begin{aligned}
C_2(1,x,z,w)
&=\theta^{q^2}\left(\frac{\Delta_{013}}{\Delta_{012}}\right)^{2q_0}
 +\theta^{q^2+q}\left(\frac{\Delta_{023}}{\Delta_{012}}\right)^{2q_0}\\
&=\theta^{q^2}\left(1+\theta^q+\theta^{q+q_0}\right)^{2q_0}
 +\theta^{q^2+q}\left(\theta^q+\theta^{q+q_0}+\theta^{q+q_0+1}\right)^{2q_0}\\
&=\theta^{q^2}\left(1+\theta^q+\theta^{q+2q_0}+\theta^{q+2q_0+1}\right)^{2q_0}
 =\theta^{q^2}B_S(\theta)^{2q_0}.
\end{aligned}
\]
Thus
\begin{equation}\label{eq:suzuki-C-values}
C_3(1,x,z,w)=B_S(\theta)^q,\qquad C_2(1,x,z,w)=\theta^{q^2}B_S(\theta)^{2q_0}.
\end{equation}
Consequently,
\[
U_S(1,x,z,w)=\frac{C_3^{q+1}}{C_2^q}
=\frac{B_S(\theta)^{q(q+1)}}{\theta^{q^3}B_S(\theta)^{2q_0q}}
=\left(\frac{B_S(\theta)^{q-2q_0+1}}{\theta^{q^2}}\right)^q=t_S^q.
\]
This calculation explains the formula for $t_S$.

We now prove directly that $t_S$ is fixed by the generators of $\Sz(q)$.

\begin{lemma}\label{lem:suzuki-stabilizer}
The function $t_S$ is fixed by the stabilizer of $P_\infty$ in $\Sz(q)$.
\end{lemma}

\begin{proof}
For $\tau_{a,c}$ with $a,c\in\F_q$, we have
\[
        (x+a)^q+(x+a)=x^q+x=A.
\]
Thus $A$, $\theta$, and $t_S$ are fixed by $\tau_{a,c}$.  For $\psi_d$ with $d\in\F_q^*$,
\[
        (dx)^q+dx=dA.
\]
Therefore
\[
        \theta\mapsto (dA)^{q-1}=d^{q-1}\theta=\theta,
\]
and $t_S$ is fixed by $\psi_d$.
\end{proof}

\begin{lemma}\label{lem:suzuki-theta-iota}
Under the involution $\iota$ one has
\[
        \iota(A)=\frac{A}{w^m},
        \qquad
        \iota(\theta)=\frac{\theta}{\lambda^m},
\]
where $m=q-2q_0+1$ and $\lambda=w^{q-1}$.
\end{lemma}

\begin{proof}
Since $\iota(x)=z/w$,
\[
        \iota(A)=\left(\frac zw\right)^q+\frac zw
        =\frac{z^qw+zw^q}{w^{q+1}}.
\]
Using \eqref{eq:suzuki-basic-relations},
\[
\begin{aligned}
        z^qw+zw^q
        &=(z+x^{2q_0}A)w+z(w+y^{2q_0}A)\\
        &=A(x^{2q_0}w+zy^{2q_0}).
\end{aligned}
\]
Now
\[
\begin{aligned}
        x^{2q_0}w+zy^{2q_0}
        &=x^{2q_0}(xy^{2q_0}+z^{2q_0})+zy^{2q_0}\\
        &=(x^{2q_0+1}+z)y^{2q_0}+x^{2q_0}z^{2q_0}\\
        &=y^{4q_0}+x^{2q_0}z^{2q_0}
        =w^{2q_0},
\end{aligned}
\]
where the last equality follows from $w=y^2+xz$.  Therefore
\[
        z^qw+zw^q=Aw^{2q_0},
\]
and
\[
        \iota(A)=\frac{Aw^{2q_0}}{w^{q+1}}=\frac{A}{w^{q-2q_0+1}}=\frac{A}{w^m}.
\]
Taking $(q-1)$-st powers gives
\[
        \iota(\theta)=\iota(A)^{q-1}=\frac{A^{q-1}}{w^{m(q-1)}}=\frac{\theta}{\lambda^m}.
\]
\end{proof}

\begin{lemma}\label{lem:suzuki-key}
In $\K(\cS_q)$ one has
\begin{equation}\label{eq:suzuki-key}
        B_S\left(\frac{\theta}{\lambda^m}\right)=\frac{B_S(\theta)}{\lambda^{q^2}},
\end{equation}
where $m=q-2q_0+1$ and $\lambda=w^{q-1}$.
\end{lemma}

\begin{proof}
We use the identity for $C_3$ established in \eqref{eq:suzuki-C-values}.
Since the projective transformation representing $\iota$ is induced by a
matrix in $\GL(4,q)$ and $C_3$ is a Dickson invariant of $\GL(4,q)$,
\[
        C_3(1,x,z,w)=C_3(w,z,x,1).
\]
Moreover, as $C_3$ is homogeneous of degree
\(
        q^4-q^3=q^3(q-1),
\)
we have

\begin{equation*}
C_3(w,z,x,1)
=w^{q^3(q-1)}
  C_3\left(1,\frac zw,\frac xw,\frac1w\right)=w^{q^3(q-1)}B_S\bigl(\iota(\theta)\bigr)^q.
\end{equation*}

By Lemma~\ref{lem:suzuki-theta-iota},
$\iota(\theta)=\theta/\lambda^m$.  Combining this equality with \eqref{eq:suzuki-C-values} gives
\[
        B_S(\theta)^q
        =w^{q^3(q-1)}
          B_S\left(\frac{\theta}{\lambda^m}\right)^q
        =\lambda^{q^3}
          B_S\left(\frac{\theta}{\lambda^m}\right)^q,
\]
which can be written as
\[
        B_S(\theta)^q
        =\left(
          \lambda^{q^2}B_S\left(\frac{\theta}{\lambda^m}\right)
          \right)^q.
\]
Therefore
\[
        B_S(\theta)
        =\lambda^{q^2}B_S\left(\frac{\theta}{\lambda^m}\right),
\]
which proves the claim.
\end{proof}

\begin{proposition}\label{prop:suzuki-invariant}
The function $t_S$ is fixed by $\Sz(q)$.
\end{proposition}

\begin{proof}
By Lemma \ref{lem:suzuki-stabilizer}, $t_S$ is fixed by the stabilizer of $P_\infty$.  Since $\Sz(q)$ is generated by this stabilizer and $\iota$, it remains to prove that $\iota(t_S)=t_S$.  Lemmas \ref{lem:suzuki-theta-iota} and \ref{lem:suzuki-key} give
\[
        \iota(\theta)=\frac{\theta}{\lambda^m},
        \qquad
        B_S(\iota(\theta))=B_S\left(\frac{\theta}{\lambda^m}\right)=\frac{B_S(\theta)}{\lambda^{q^2}}.
\]
Therefore
\[
        \iota(t_S)
        =\frac{B_S(\iota(\theta))^m}{\iota(\theta)^{q^2}}
        =\frac{(B_S(\theta)/\lambda^{q^2})^m}{(\theta/\lambda^m)^{q^2}}
        =\frac{B_S(\theta)^m}{\theta^{q^2}}=t_S.
\]
Thus $t_S$ is fixed by $\Sz(q)$.
\end{proof}

\begin{theorem}\label{thm:suzuki-invariant}
One has
\[
        \K(\cS_q)^{\Sz(q)}=\K(t_S).
\]
\end{theorem}

\begin{proof}
By Proposition \ref{prop:suzuki-invariant}, $t_S\in\K(\cS_q)^{\Sz(q)}$.  Let
\[
        R_S(T)=\frac{B_S(T)^m}{T^{q^2}}.
\]
The degree of $B_S$ is $n=q+2q_0+1$, and $mn=q^2+1$.  Therefore the numerator $B_S(T)^m$ has degree $q^2+1$, while the denominator has degree $q^2$.  Since $B_S(0)=1$, numerator and denominator are coprime and hence
\[
        \deg R_S=q^2+1.
\]
The function $x^q+x$ has pole divisor $q^2P_\infty$ and has $q^2$ simple affine zeros, namely the points of $\cS_q(\F_q)\setminus\{P_\infty\}$.  Therefore, the degree of the map $\theta=(x^q+x)^{q-1}$ is
\[
        \deg\theta=q^2(q-1).
\]
Thus
\[
        \deg(t_S)=\deg(R_S\circ\theta)=(q^2+1)q^2(q-1)=|\Sz(q)|.
\]
Therefore
\[
        [\K(\cS_q):\K(t_S)]=|\Sz(q)|.
\]
Artin's theorem \cite[Theorem~3.4.1]{Stichtenoth} gives
\[
        [\K(\cS_q):\K(\cS_q)^{\Sz(q)}]=|\Sz(q)|.
\]
Since $\K(t_S)\subseteq\K(\cS_q)^{\Sz(q)}$, the two fields are equal.
\end{proof}

\section{An explicit invariant for the Ree group}

We now apply the same general strategy to the Ree group.  Set
\[
        A=x^q-x,\qquad
        \vartheta=A^{q-1},\qquad
        \mu=q-3q_0+1.
\]
Let
\begin{equation}\label{eq:ree-B}
        B_R(T)=1+T^{E_1}+T^{E_2}-T^{E_3}
        +T^{E_4}+T^{E_5}+T^{E_6},
\end{equation}
where the exponents $E_i$ are those in \eqref{eq:ree-E}.  Explicitly,
\[
\begin{aligned}
B_R(T)={}&1+T^{q^2}+T^{q^2+3q_0q}
          -T^{q^2+3q_0q+q}+T^{q^2+3q_0q+2q}\\
&+T^{q^2+3q_0q+2q+3q_0}
 +T^{q^2+3q_0q+2q+3q_0+1}.
\end{aligned}
\]
We define
\begin{equation}\label{eq:ree-t}
        t_R=\frac{B_R(\vartheta)^\mu}{\vartheta^{q^3}}.
\end{equation}
We prove in this section that
\[
        \K(\cR_q)^{\Ree(q)}=\K(t_R).
\]

For the proof, we use the quadratic form of the
seven-dimensional Ree representation.  Put
\[
 r=3q_0,\qquad
 \alpha=(1,x,w_1,w_2,w_3,w_6,w_8),\qquad
 \beta=(0,1,x,y,z,w_4,w_7).
\]
Vector powers are taken componentwise.  For row vectors
\[
        U=(U_0,\ldots,U_6),\qquad V=(V_0,\ldots,V_6),
\]
define
\begin{equation}\label{eq:ree-pairing}
\begin{aligned}
 \langle U,V\rangle={}&U_0V_6+U_6V_0
 +U_1V_5+U_5V_1\\
 &+U_2V_4+U_4V_2+U_3V_3 .
\end{aligned}
\end{equation}
This is the symmetric bilinear form associated with the standard quadratic
form of the Ree representation; see \cite[Section~2.1]{EidDuursma}.  The
Frobenius-difference identities \eqref{eq:ree-frob-differences} give the
vector identity
\begin{equation}\label{eq:ree-alpha-difference}
        \alpha^q-\alpha=A\beta^r.
\end{equation}
For $i\geq0$, set
\begin{equation}\label{eq:ree-FH}
        F_i=\langle\alpha,\alpha^{q^i}\rangle,
        \qquad
        H_i=\langle\beta,\beta^{q^i}\rangle.
\end{equation}

\begin{lemma}\label{lem:ree-H-values}
Let $L=q+2q_0+1$.  Then
\[
        H_0=H_1=0,\qquad H_2=-A^L,
\]
and
\begin{equation}\label{eq:ree-H3}
\begin{aligned}
H_3={}&-A^{q+2q_0+1}
       -A^{q^2+2q_0+1}
       +A^{q^2+qq_0+q_0+1}\\
&\quad -A^{q^2+2qq_0+1}
       -A^{q^2+2qq_0+q}.
\end{aligned}
\end{equation}
\end{lemma}

\begin{proof}
Put $u=w_2-xw_1$.  From the definitions of $w_1$ and $w_2$,
\begin{equation}\label{eq:ree-u-q0}
        u^{q_0}=-x^{q_0}y^q-z^q-x^{q+2q_0}.
\end{equation}
Moreover,
\begin{equation}\label{eq:ree-w4-identities}
        w_4=y^2-xz,\qquad
        w_4^q-w_4=Au^{q_0}.
\end{equation}
For completeness, the first identity follows after substituting
\[
        w_2^{q_0}=x^{q_0}y^q-z^q,\qquad
        w_1^{q_0}=x^{q+q_0}-y^q
\]
in the definition of $w_4$.  The second then follows from
\eqref{eq:ree-curve}: indeed,
\[
\begin{aligned}
(y^2-xz)^q-(y^2-xz)
 &=A\bigl(-x^{q_0}y-z-x^{2q_0+1}\bigr)\\
 &=Au^{q_0},
\end{aligned}
\]
where the last equality is obtained by substituting the two equations in
\eqref{eq:ree-curve} into \eqref{eq:ree-u-q0}.

The identity $H_0=0$ follows from $w_4=y^2-xz$ and
characteristic $3$.  Next,
\[
        H_1=w_4^q+w_4+xz^q+x^qz+y^{q+1}.
\]
Using \eqref{eq:ree-w4-identities} and then \eqref{eq:ree-curve}, this
becomes
\[
        H_1=A\bigl(u^{q_0}+x^{2q_0+1}+z+x^{q_0}y\bigr)=0
\]
by \eqref{eq:ree-u-q0}.

Since
\[
        H_1^q-H_2
        =\langle\beta^q-\beta,\beta^{q^2}\rangle,
\]
the two equations in \eqref{eq:ree-curve} and
\eqref{eq:ree-w4-identities} yield
\begin{equation}\label{eq:ree-H2-difference}
\begin{aligned}
H_1^q-H_2
 &=A\bigl(u^{q_0}+z^{q^2}+x^{q^2+2q_0}
              +x^{q_0}y^{q^2}\bigr)\\
 &=A^{q+1}\bigl(x^{2q_0q}
              +x^{q_0(q+1)}+x^{2q_0}\bigr)\\
 &=A^{q+2q_0+1}.
\end{aligned}
\end{equation}
For the last equality we used
\[
 x^{2q_0q}+x^{q_0(q+1)}+x^{2q_0}
 =(x^q-x)^{2q_0}=A^{2q_0}.
\]
Since $H_1=0$, formula \eqref{eq:ree-H2-difference} gives
$H_2=-A^L$.

Similarly,
\begin{equation}\label{eq:ree-H3-difference}
        H_2^q-H_3=AC,
\end{equation}
where
\[
        C=u^{q_0}+z^{q^3}+x^{q^3+2q_0}
          +x^{q_0}y^{q^3}.
\]
Iterating the two equations in \eqref{eq:ree-curve} gives
\begin{align*}
C={}&A^q\bigl(x^{2q_0q}+x^{q_0(q+1)}+x^{2q_0}\bigr)\\
&+A^{q^2}\bigl(x^{2q_0q^2}
               +x^{q_0(q^2+1)}+x^{2q_0}\bigr)\\
={}&A^{q+2q_0}+A^{q^2}(A^q+A)^{2q_0}\\
={}&A^{q+2q_0}+A^{q^2+2q_0}
     -A^{q^2+q_0(q+1)}+A^{q^2+2q_0q}.
\end{align*}
Substitution in \eqref{eq:ree-H3-difference}, together with
$H_2^q=-A^{qL}$, proves \eqref{eq:ree-H3}.
\end{proof}

\begin{lemma}\label{lem:ree-low-F}
One has
\[
        F_0=F_1=F_2=0.
\]
\end{lemma}

\begin{proof}
The standard quadratic equation of the Ree ovoid is
\begin{equation}\label{eq:ree-quadric}
        w_8+xw_6+w_1w_3-w_2^2=0;
\end{equation}
see \cite[Section~4]{EidDuursma}.  Since the
characteristic is $3$, equation \eqref{eq:ree-quadric} is precisely
$F_0=0$.

We also need the mixed pairing
\begin{equation}\label{eq:ree-mixed-zero}
        \langle\beta^r,\alpha\rangle=0.
\end{equation}
We verify it directly.  Write its left hand side as
\[
        M=w_7^r+w_6+xw_4^r+x^rw_3+z^rw_1+y^rw_2.
\]
Using $w_7=w_2+v$ and
$w_6=v^r-w_2^r+xw_4^r$, and recalling that $2=-1$, gives
\[
        M=-v^r-xw_4^r+x^rw_3+z^rw_1+y^rw_2.
\]
Now
\[
        v^r=x^rw_3^q-z^rw_1^q,
\]
and \eqref{eq:ree-frob-differences} therefore gives
\[
        M=-z^rw_1-xw_4^r+y^rw_2.
\]
Finally, if
\[
        T=xw_4^r+z^rw_1-y^rw_2,
\]
then the definitions of $w_1,w_2,w_4$ give
\begin{align*}
T
&=x^{r+1}w_2^q-xy^rw_1^q+z^rw_1-y^rw_2\\
&=(w_1+y^r)(w_2^q-w_2)-xy^r(w_1^q-w_1)\\
&=x^{r+1}y^rA-x^{r+1}y^rA=0.
\end{align*}
Thus $M=-T=0$, proving \eqref{eq:ree-mixed-zero}.

From \eqref{eq:ree-alpha-difference},
\[
        F_1=F_0+A\langle\alpha,\beta^r\rangle=0.
\]
Furthermore,
\[
        \alpha^{q^2}=\alpha+A\beta^r+A^q\beta^{rq}.
\]
Pairing this identity with $\beta^r$ and using
\eqref{eq:ree-mixed-zero} and $H_0=H_1=0$, we obtain
\[
        \langle\beta^r,\alpha^{q^2}\rangle=0.
\]
Since
\[
F_1^q=\langle\alpha^q,\alpha^{q^2}\rangle
      =F_2+A\langle\beta^r,\alpha^{q^2}\rangle,
\]
it follows that $F_2=0$.
\end{proof}

\begin{proposition}\label{prop:ree-bilinear-identity}
In $\K(\cR_q)$ one has
\begin{equation}\label{eq:ree-F3-F4}
        F_3=A^{E_6},\qquad
        F_4=A^{E_6}B_R(\vartheta).
\end{equation}
In particular,
\begin{equation}\label{eq:ree-B-pairing}
        B_R(\vartheta)=\frac{F_4}{F_3}.
\end{equation}
\end{proposition}

\begin{proof}
Iterating \eqref{eq:ree-alpha-difference} gives
\begin{equation}\label{eq:ree-alpha-iterate}
        \alpha^{q^h}
        =\alpha+\sum_{j=0}^{h-1}A^{q^j}\beta^{rq^j}.
\end{equation}
Since the coefficients of the pairing lie in the prime field $\F_3$, we have
\[
        \langle\beta^r,\beta^{rq^j}\rangle=H_j^r.
\]
Consequently, \eqref{eq:ree-mixed-zero} and
\eqref{eq:ree-alpha-iterate} give the recurrence
\begin{equation}\label{eq:ree-F-recurrence}
\begin{aligned}
F_h-F_{h-1}^q
 &=-A\langle\beta^r,\alpha^{q^h}\rangle\\
 &=-A\sum_{j=1}^{h-1}A^{q^j}H_j^r.
\end{aligned}
\end{equation}
For $h=3$, Lemmas~\ref{lem:ree-H-values} and \ref{lem:ree-low-F}
give
\[
F_3=-A\,A^{q^2}H_2^r
   =A^{1+q^2+rL}
   =A^{E_6},
\]
because $1+q^2+rL=E_6$.

For $h=4$, the same recurrence gives
\begin{equation}\label{eq:ree-F4-before}
        F_4=A^{qE_6}+A^{E_6}-A^{q^3+1}H_3^r.
\end{equation}
Let the five exponents occurring in \eqref{eq:ree-H3} be denoted by
$d_1,\ldots,d_5$.  The correspondence with the exponents in
\eqref{eq:ree-B} is
\[
\begin{array}{c|c|c}
j&d_j&(q^3+1+rd_j-E_6)/(q-1)\\ \hline
1&q+2q_0+1&q^2=E_1\\
2&q^2+2q_0+1&q^2+3q_0q=E_2\\
3&q^2+qq_0+q_0+1&q^2+3q_0q+q=E_3\\
4&q^2+2qq_0+1&q^2+3q_0q+2q=E_4\\
5&q^2+2qq_0+q&q^2+3q_0q+2q+3q_0=E_5 .
\end{array}
\]
Each entry in the last column follows from $q=3q_0^2$ and
$r=3q_0$.  Thus
\begin{equation}\label{eq:ree-exponent-match}
 q^3+1+rd_j=E_6+(q-1)E_j\quad(1\leq j\leq5),
 \qquad
 qE_6=E_6+(q-1)E_6.
\end{equation}
Since $r$ is a power of $3$, raising \eqref{eq:ree-H3} to the
$r$-th power raises each monomial separately and preserves its sign.
Substituting \eqref{eq:ree-H3} and \eqref{eq:ree-exponent-match}
in \eqref{eq:ree-F4-before}, and recalling that
$\vartheta=A^{q-1}$, yields
\[
\begin{aligned}
F_4
 =A^{E_6}\bigl(&1+\vartheta^{E_1}+\vartheta^{E_2}
 -\vartheta^{E_3}+\vartheta^{E_4}\\
 &+\vartheta^{E_5}+\vartheta^{E_6}\bigr)
 =A^{E_6}B_R(\vartheta).
\end{aligned}
\]
This proves \eqref{eq:ree-F3-F4} and \eqref{eq:ree-B-pairing}.  Notice
that the positive term in $H_3$ is exactly what produces the minus sign
of the term $T^{E_3}$ in $B_R(T)$.
\end{proof}

\begin{lemma}\label{lem:ree-stabilizer}
The function $t_R$ is fixed by the stabilizer of $P_\infty$ in $\Ree(q)$.
\end{lemma}

\begin{proof}
An element of the stabilizer sends $x$ to $ax+b$, where
$a\in\F_q^*$ and $b\in\F_q$.  Hence
\[
        (ax+b)^q-(ax+b)=a(x^q-x)=aA.
\]
Therefore
\[
        \vartheta\mapsto(aA)^{q-1}
        =a^{q-1}\vartheta=\vartheta.
\]
Since $t_R$ is a rational function of $\vartheta$, it is fixed by the
stabilizer.
\end{proof}

Put
\[
        \Lambda=w_8^{q-1}.
\]

\begin{lemma}\label{lem:ree-A-phi}
Under the involution $\phi$ one has
\[
        \phi(A)=\frac{A}{w_8^\mu},
        \qquad
        \phi(\vartheta)=\frac{\vartheta}{\Lambda^\mu}.
\]
\end{lemma}

\begin{proof}
Since $\phi(x)=w_6/w_8$,
\[
\phi(A)=\left(\frac{w_6}{w_8}\right)^q-\frac{w_6}{w_8}
       =\frac{w_6^qw_8-w_6w_8^q}{w_8^{q+1}}.
\]
Using \eqref{eq:ree-frob-differences}, the numerator becomes
\[
\begin{aligned}
w_6^qw_8-w_6w_8^q
 &=(w_6+w_4^{3q_0}A)w_8-w_6(w_8+w_7^{3q_0}A)\\
 &=A(w_4^{3q_0}w_8-w_6w_7^{3q_0}).
\end{aligned}
\]
The standard Ree relation
\begin{equation}\label{eq:ree-w8-relation}
        w_8^{3q_0}=w_8w_4^{3q_0}-w_6w_7^{3q_0},
\end{equation}
used in the proof of \cite[Lemma~4.2]{Skabelund}, therefore gives
\[
        w_6^qw_8-w_6w_8^q=Aw_8^{3q_0}.
\]
Hence
\[
        \phi(A)=\frac{A}{w_8^{q-3q_0+1}}
               =\frac{A}{w_8^\mu}.
\]
Taking $(q-1)$-st powers yields
\[
        \phi(\vartheta)
        =\phi(A)^{q-1}
        =\frac{\vartheta}{\Lambda^\mu}.
\]
\end{proof}

\begin{lemma}\label{lem:ree-key}
In $\K(\cR_q)$ one has
\begin{equation}\label{eq:ree-key}
        B_R\left(\frac{\vartheta}{\Lambda^\mu}\right)
        =\frac{B_R(\vartheta)}{\Lambda^{q^3}}.
\end{equation}
\end{lemma}

\begin{proof}
Let
\[
        J(U_0,U_1,U_2,U_3,U_4,U_5,U_6)
        =(U_6,U_5,U_4,-U_3,U_2,U_1,U_0).
\]
The coordinate formulas for $\phi$ say precisely that
\[
        \phi(\alpha)=\frac{J\alpha}{w_8}.
\]
Moreover, $J$ preserves the pairing \eqref{eq:ree-pairing}; indeed,
the middle term $U_3V_3$ is unchanged by the simultaneous change
of sign in the two middle coordinates.  Since $J$ has coefficients
in $\F_3$, it commutes with the componentwise Frobenius powers.
Hence, for every $i\geq0$,
\begin{equation}\label{eq:ree-Fi-phi}
        \phi(F_i)
        =\frac{\langle J\alpha,J\alpha^{q^i}\rangle}
               {w_8^{q^i+1}}
        =\frac{F_i}{w_8^{q^i+1}}.
\end{equation}
By Proposition~\ref{prop:ree-bilinear-identity} and
Lemma~\ref{lem:ree-A-phi},
\[
\begin{aligned}
B_R\left(\frac{\vartheta}{\Lambda^\mu}\right)
 &=B_R(\phi(\vartheta))
  =\phi\left(\frac{F_4}{F_3}\right)\\
 &=\frac{F_4}{F_3}\,w_8^{q^3-q^4}
  =\frac{B_R(\vartheta)}{w_8^{q^3(q-1)}}
  =\frac{B_R(\vartheta)}{\Lambda^{q^3}}.
\end{aligned}
\]
This proves \eqref{eq:ree-key}.
\end{proof}

\begin{proposition}\label{prop:ree-invariant}
The function $t_R$ is fixed by $\Ree(q)$.
\end{proposition}

\begin{proof}
By Lemma~\ref{lem:ree-stabilizer}, $t_R$ is fixed by the stabilizer of
$P_\infty$.  Since $\Ree(q)$ is generated by this stabilizer and $\phi$, it
remains to prove that $\phi(t_R)=t_R$.  Lemmas~\ref{lem:ree-A-phi} and
\ref{lem:ree-key} give
\[
        \phi(\vartheta)=\frac{\vartheta}{\Lambda^\mu},
        \qquad
        B_R(\phi(\vartheta))
        =\frac{B_R(\vartheta)}{\Lambda^{q^3}}.
\]
Therefore
\[
\begin{aligned}
\phi(t_R)
 &=\frac{B_R(\phi(\vartheta))^\mu}{\phi(\vartheta)^{q^3}}\\
 &=\frac{(B_R(\vartheta)/\Lambda^{q^3})^\mu}
         {(\vartheta/\Lambda^\mu)^{q^3}}
  =\frac{B_R(\vartheta)^\mu}{\vartheta^{q^3}}
  =t_R.
\end{aligned}
\]
Thus $t_R$ is fixed by $\Ree(q)$.
\end{proof}

\begin{theorem}\label{thm:ree-invariant}
One has
\[
        \K(\cR_q)^{\Ree(q)}=\K(t_R).
\]
\end{theorem}

\begin{proof}
By Proposition~\ref{prop:ree-invariant},
$t_R\in\K(\cR_q)^{\Ree(q)}$.  Let
\[
        R_R(T)=\frac{B_R(T)^\mu}{T^{q^3}}.
\]
By \eqref{eq:ree-B}, $\deg B_R=E_6$, and
$\mu E_6=q^3+1$.  Therefore the numerator $B_R(T)^\mu$ has degree
$q^3+1$, while the denominator has degree $q^3$.  Since $B_R(0)=1$,
the numerator and denominator are coprime, and hence
\[
        \deg R_R=q^3+1.
\]
The function $x^q-x$ has pole divisor $q^3P_\infty$ and has $q^3$
simple affine zeros, namely the points of
$\cR_q(\F_q)\setminus\{P_\infty\}$.  Hence
\[
        \deg\vartheta=q^3(q-1).
\]
Thus
\[
        \deg(t_R)=\deg(R_R\circ\vartheta)
        =(q^3+1)q^3(q-1)=|\Ree(q)|.
\]
Therefore
\[
        [\K(\cR_q):\K(t_R)]=|\Ree(q)|.
\]
Artin's theorem \cite[Theorem~3.4.1]{Stichtenoth} gives
\[
        [\K(\cR_q):\K(\cR_q)^{\Ree(q)}]=|\Ree(q)|.
\]
Since $\K(t_R)\subseteq\K(\cR_q)^{\Ree(q)}$, the two fields are equal.
\end{proof}

We close with an application of the two invariants to the construction of
covers on which the full automorphism group still acts.

\begin{remark}\label{rem:prank}
Let $F$ be either $\K(\cS_q)$ or $\K(\cR_q)$, let $G=\Aut(F)$ be the
corresponding full automorphism group, and let $t$ be the invariant $t_S$ or
$t_R$, so that $F^G=\K(t)$ by Theorem~\ref{thm:suzuki-invariant} or
Theorem~\ref{thm:ree-invariant}.  An explicit generator of $F^G$ gives a
general way of writing extensions of $F$ which retain the action of $G$.
Indeed, for $R(T)\in\K(T)$ the extensions defined by
\[
        u^n=R(t),\qquad p\nmid n,\quad n>1,
        \qquad\text{or}\qquad
        u^p-u=R(t),
\]
allow every element of $G$ to lift by fixing $u$, provided the defining
polynomial is irreducible over $F$.  The lifted group commutes with the cyclic
covering group, so the automorphism group of the extension contains
$G\times C_n$ or $G\times C_p$, respectively. 

The construction also produces explicit covers of positive $p$-rank.  Choose
$a\in\K$ outside the branch locus of the quotient map defined by $t$, and put
\[
        E=F(u),\qquad u^p-u=\frac{1}{t-a}.
\]
The right-hand side has simple poles at the $|G|$ points above $a$.  Hence the
defining polynomial is irreducible over $F$, and $E/F$ is ramified precisely at
these points, each of which is totally ramified.  Both $\cS_q$ and $\cR_q$ are
maximal, hence supersingular, and therefore have $p$-rank zero; the
Deuring--Shafarevich formula (see \cite{HKT}) then gives
\[
        \gamma(E)-1=p(0-1)+|G|(p-1),
        \qquad\text{that is,}\qquad
        \gamma(E)=(p-1)(|G|-1)>0,
\]
where $\gamma$ denotes the $p$-rank.  Allowing simple poles above several
distinct unramified fibres gives covers of unbounded genus and $p$-rank which
still carry the action of $G$.  These positive-$p$-rank covers are of course
not maximal curves; they provide a separate application of the explicit
generators determined here to the construction of curves with Suzuki or Ree
symmetries.  For the extension displayed above, the exhibited subgroup of
$\Aut(E)$ has order
\[
        p|G|=\frac{p}{p-1}\,\gamma(E)+p,
\]
so its order grows only linearly with the $p$-rank.  It would be interesting to
look for invariant covers in which the $p$-rank grows more slowly relative to
the lifted group, and to compare them, when the relevant hypotheses hold, with
the quartic $p$-rank bound of \cite{GKTPRank}.
\end{remark}

\section*{Acknowledgements}
The authors thank the Italian National Group for Algebraic and Geometric Structures and their Applications (GNSAGA—INdAM)
which supported the research.

\end{document}